\documentclass{article}
\usepackage{amsmath}
\usepackage{amsfonts}
\usepackage{amssymb}
\usepackage{graphicx}

\newtheorem{theorem}{Theorem}[section]

\newtheorem{corollary}[theorem]{Corollary}

\newtheorem{definition}[theorem]{Definition}
\newtheorem{example}[theorem]{Example}

\newtheorem{lemma}[theorem]{Lemma}

\newtheorem{proposition}[theorem]{Proposition}
\newtheorem{remark}[theorem]{Remark}

\newenvironment{proof}[1][Proof]{\noindent \textbf{#1.} }{\  \rule{0.5em}{0.5em}}
\numberwithin{equation}{section}

\newcommand{\Real}{\mathbb R}

\newcommand{\oo}{\omega}

\begin{document}

\title{Jensen's Inequality for $g$-Convex Function under $g$-Expectation}
\author{Guangyan JIA\thanks{%
The author thanks the partial support from The National Basic Research
Program of China (973 Program) grant No. 2007CB814901 (Financial Risk).
email address: jiagy@sdu.edu.cn} \ and Shige PENG\thanks{%
The author thanks the partial support from The National Basic Research
Program of China (973 Program) grant No. 2007CB814900 (Financial Risk).
email address: peng@sdu.edu.cn} \\
%EndAName
School of Mathematics and System Sciences\\
Shandong University, Jinan, Shandong, 250100, P.R.China}
\maketitle

{\small \textbf{Abstract.} A real valued function defined on }$\mathbb{R}$
{\small is called }$g${\small --convex if it satisfies the following
\textquotedblleft generalized Jensen's inequality\textquotedblright \ under
a given }$g${\small -expectation, i.e., }$h(\mathbb{E}^{g}[X])\leq \mathbb{E}%
^{g}[h(X)]${\small , for all random variables }$X$ {\small such that both
sides of the inequality are meaningful. In this paper we will give a
necessary and sufficient conditions for a }$C^{2}${\small -function being }$%
g ${\small -convex. We also studied some more general situations. We also
studied }$g${\small -concave and }$g${\small -affine functions.}

\section{Introduction}

Jensen's inequality plays an important role in probability theory. It claims
that\ for any given convex function $h$ defined on $\mathbb{R}$ we have%
\begin{equation*}
h(E[X])\leq E[h(X)]
\end{equation*}%
for each random variable $X$ such that $E[X]$ and $E[h(X)]$ are meaningful.
Here $E[\cdot ]$ stands for the expectation related to a probability $P$. It
is worth to mention that its converse is also true: If the above inequality
holds true for all random variables $X$ such that both $E[X]$ and $E[h(X)]$
are meaningful, then $h$ is a convex function.

In 1997 Peng \cite{peng1997} (see also \cite{peng1997b}) introduced the
notion of $g$-expectation $\mathbb{E}^{g}[\cdot ]$ defined via a backward
stochastic differential equation of which the generator is a given function $%
g=g(t,y,z)_{(t,y,z)\in \lbrack 0,T]\times \mathbb{R}\times \mathbb{R}^{d}}$.
A $g$-expectation preserves most properties of the classical expectations
except that it is a nonlinear functional. Its nonlinearity is characterized
by its generator $g$. It becomes a typical example of nonlinear expectations
under which the time-consistency holds true thus a theory of nonlinear
martingales can be developed. It is also a useful tool to the nonlinear
dynamic pricing as well as dynamic risk measures in finance.

A very interesting problem is whether, for a $g$-expectation, the following
generalized Jensen's inequality is true:%
\begin{align*}
h(\mathbb{E}^{g}[X])& \leq \mathbb{E}^{g}[h(X)],\  \  \  \\
& \text{for each }X\text{ s.t. }\mathbb{E}^{g}[X]\text{ and }\mathbb{E}%
^{g}[h(X)]\text{ are meaningful.}
\end{align*}%
This problem was initialed in \cite[CHMP2000]{briandCHMP2000} in which a
counterexample was given to show that the above generalized Jensen's
inequality fails for a very simple convex functions $h$. A sufficient
condition for a special situation was also provided. Chen, Kulperger and
Jiang \cite[2003]{chenKJ2003} have obtained a very interesting result:
provided $g$ does not depend on $y$, the above generalized Jensen's
inequality holds true for each convex function $h$ if and only if $g$ is a
super-homogeneous function, i.e., $g(t,\lambda z)\geq \lambda g(t,z)$, $%
dP\times dt-a.s.$ for $\lambda \in \mathbb{R}$ and $z\in \mathbb{R}^d$. This
result was improved by \cite[2005]{hu2005} showing that, in fact, $g$ must
be independent of $y$.

In this paper we study this problem with a different point of view: For each
fixed function $g$, to give an explicit characterization to $h$ satisfying
the above generalized Jensen's inequality. We have obtained the following
result: For a $C^{2}$-function $h$ the above generalized Jensen inequality
holds if and only if $h$ satisfies:
\begin{equation*}
\frac{1}{2}h^{\prime \prime }(y)|z|^{2}+g(t,h(y),h^{\prime }(y)z)-h^{\prime
}(y)g(t,y,z)\geq 0,\,dP\times dt-a.s., \  \forall (y,z)\in \mathbb{R}\times
\mathbb{R}^{d}.
\end{equation*}%
The previously mentioned result of classical Jensen's inequality just
corresponds a special case where $g\equiv 0$. The above mentioned results in
\cite{chenKJ2003} and \cite{hu2005} can be also obtained from our new
result. For the case where $h$ is only a continuous function we have also
obtained a similar result by using the notion of the well-known viscosity
solution in partial differential equations.

It is natural to call a $h$ satisfying the above inequality to be a $g$%
-convex function. In general, a continuous function $h$ satisfying the
generalized Jensen's inequality is called a $g$-convex function. In this
paper we will study this type of functions. We also investigate the related $%
g$-concave as well as $g$-affine functions. A deep relation of $g$-convexity
and backward stochastic viability property introduced by Buckdahn,
Quincampoix and Rascanu in \cite{buckdahnQR2000} is also disclosed.

%\textbf{Re-organize the following}

This paper is organized as follows. In Section \ref{sec2} we recall some
facts about $g$-expectation and BSDEs. The notion of $g$-convexity as well
as the necessary and sufficient condition for a $g$--convex $C^{2}$-function
will be given in Section \ref{sec3}. We establish the necessary and
sufficient condition for a continuous $g$-convex function in Section \ref%
{sec4}. An equivalence between $g$-convexity and backward stochastic
viability property is given in Section \ref{sec5}. Finally, in
Section \ref{sec7} we study functional operations preserving
$g$-convexity and apply the results obtained in foregoing sections
to prove some properties of $g$-expectations.

%\bigskip \newpage

\section{Some Facts about $g$-Expectations}

\label{sec2}

Let $(\Omega,\mathcal{F},P)$ be a given probability space and let $%
(W_{t})_{t\geq0}$ be a $d$--dimensional Brownian motion in this space. The
natural filtration generated by $W$ will be denoted by $(\mathcal{F}%
_{t})_{t\geq0}$.

Let $T>0$ be a fixed real number. For any $0\leq t\leq T$, we denote by $%
L^{p}(\mathcal{F}_{t})$, the space of $\mathcal{F}_{t}$-measurable random
variables satisfying $E[|X|^{p}]<\infty $, for $p\geq 1$. For a positive
integer $n$ and $z\in \mathbb{R}^{n}$, we denote by $\left \vert
z\right
\vert $ the Euclidean norm of $z$. We will denote by $L_{\mathcal{F}%
}^{2}(0,T;\mathbb{R}^{n})$, the space of all progressively measurable $%
\mathbb{R}^{n}$--valued processes such that $\mathbb{E}\left[
\int_{0}^{T}\left \vert \psi _{t}\right \vert ^{2}\,dt\right] <\infty $; and
by $\mathcal{S}_{\mathcal{F}}^{2}(0,T;\mathbb{R}^{n})$ the elements in $L_{%
\mathcal{F}}^{2}(0,T;\mathbb{R}^{n})$ with continuous paths such that $%
\mathbb{E}\left[ \sup_{t\in \lbrack 0,T]}\left \vert \psi _{t}\right \vert
^{2}\right] <\infty $. And we denote by $\mathcal{D}_{\mathcal{F}}^{2}(0,T)$
the set of all RCLL (right continuous with left limit) processes $\phi $ in $%
L_{\mathcal{F}}^{2}(0,T;\mathbb{R})$ such that $E[\sup_{t\in[0,T]}\left
\vert \phi _{t}\right \vert ^{2}]<\infty $.

Let us consider a function $g$, which will be in the sequel the generator of
the backward stochastic differential equation (BSDE), defined on $\Omega
\times \lbrack 0,T]\times \mathbb{R}^{m}\times \mathbb{R}^{m\times d}$, with
values in $\mathbb{R}^{m}$, such that the process $(g(t,y,z))_{t\in \lbrack
0,T]}$ is progressively measurable for each $(y,z)$ in $\mathbb{R}^{m}\times
\mathbb{R}^{m\times d}$. Through out this paper the function $g$ will
satisfy the following conditions..
\begin{equation}
\begin{cases}
(a) & \text{There exists a constant }\mu >0\text{, for each }(y,z),(\bar{y},%
\bar{z})\in \mathbb{R}^{m}\times \mathbb{R}^{m\times d}, \\
& \qquad \left \vert g(t,y,z)-g(t,\bar{y},\bar{z})\right \vert \leq \mu
(\left \vert y-\bar{y}\right \vert +\left \vert z-\bar{z}\right \vert ); \\
(b) & (g(t,0,0))_{t\in \lbrack 0,T]}\in L_{\mathcal{F}}^{\infty }(0,T;%
\mathbb{R}^{m}).%
\end{cases}
\label{eqn2.2}
\end{equation}%
It is by now well known (see Pardoux and Peng \cite{pardouxP1990}) that
under the assumptions (\ref{eqn2.2}), for any random variable $X\in L^{2}(%
\mathcal{F}_{T})$, the BSDE
\begin{equation}
Y_{t}=X+\int_{t}^{T}g(s,Y_{s},Z_{s})\,ds-\int_{t}^{T}Z_{s}\,dW_{s},\quad
t\in \lbrack 0,T],  \label{eqn2.1}
\end{equation}%
has a unique adapted solution $(Y_{t},Z_{t})_{t\in \lbrack 0,T]}\in \mathcal{%
S}_{\mathcal{F}}^{2}(0,T;\mathbb{R}^{m})\times L_{\mathcal{F}}^{2}(0,T;%
\mathbb{R}^{m\times d})$. In the sequel we denote equation (\ref{eqn2.1}) by
$(g,T,X)$.

In this paper we mainly discuss the $1$-dimensional BSDE, i.e., $m=1$. The
following situations are typical:
\begin{equation}
\begin{cases}
(a).\quad g(\cdot ,0,0)\equiv 0, \\
(b).\quad g(\cdot ,y,0)\equiv 0,\forall y\in \mathbb{R}.%
\end{cases}
\label{eqn2.3}
\end{equation}%
Obviously (b) implies (a). The following notion of $g$-expectation was
introduced by Peng \cite{peng1997}.

\begin{definition}
\label{defn2.1} Let $m=1$. We denote by $\mathbb{E}_{t,T}^g[X]:=Y_t$:
\begin{equation}  \label{eqn2.4}
\mathbb{E}_{t,T}^g[\cdot]: L^2(\mathcal{F}_T)\to L^2(\mathcal{F}_t),\quad
0\le t\le T<\infty.
\end{equation}
$(\mathbb{E}_{t,T}^g[\cdot])_{0\le t\le T}$ is called $g$-expectation.
\end{definition}

Applications of $g$-expectations in dynamic superpricing and dynamic risk
measures can be found in \cite{barrieuE2005, chenE2002, coquetHMP2002,
duffieE1992, elkarouiPQ1997, elkarouiQ1997, frittelliR2004, peng2004,
peng2004b, rosazza2005, yong1999}.

\begin{remark}
\label{rem2.11} The $g$-expectation originally introduced in \cite{peng1997}
corresponds the case in which $g$ satisfies (\ref{eqn2.3})-b, that is, the
situation of "zero interest rate" (see next section or \cite{peng2004}).
Peng \cite{peng2004}, \cite{peng2005} also introduced the notion of $g$%
-evaluation if $g$ satisfies (\ref{eqn2.3})-a, the situation of
"self-financing" . For the simplicity, we call them all $g$-expectation here
whenever $g$ satisfies (\ref{eqn2.2}).
\end{remark}

Also we have the following properties about $g$-expectation (see \cite[%
Theorem 3.4]{peng2004}):

\begin{proposition}
\label{prop2.4} Let the generator $g$ satisfies (\ref{eqn2.2}) and (\ref%
{eqn2.3})-a. Then the above defined $g$-expectation $\mathbb{E}^g[\cdot]$
satisfies, for each $t\le T<\infty$, $X,\bar{X}\in L^2(\mathcal{F}_T)$,

(\textbf{A1}) $\mathbb{E}_{t,T}^g[X]\ge \mathbb{E}^g_{t,T}[\bar{X}]$, a.s.,
if $X\ge \bar{X}$;

(\textbf{A2}) $\mathbb{E}_{T,T}^g[X]=X$;

(\textbf{A3}) $\mathbb{E}_{s,t}^g[\mathbb{E}_{t,T}^g[X]]=\mathbb{E}%
^g_{s,T}[X]$, a.s., for $s\le t$;

(\textbf{A4}) $1_{A}\mathbb{E}_{t,T}^{g}[X]=\mathbb{E}_{t,T}^{g}[1_{A}X]$, $%
\forall A\in \mathcal{F}_{t}$,\newline
where $1_{A}$ is the indicator function of $A$, i.e. $1_{A}(\omega )$ equals
$1$ when $\omega \in A$ and $0$ otherwise.

If (\ref{eqn2.3}) does not hold, (A1)-(A3) still hold true. But (A4) is
replaced by the following: for each $X_{1},\cdots ,X_{N}\in L^{2}(\mathcal{F}%
_{t})$ and for each $\mathcal{F}_{t}$-partition $\left \{ A_{i}\right \}
_{i=1}^{N}$ of $\Omega$ (i.e. $A_{i}\in \mathcal{F}_{t}$, $A_{i}\cap
A_{j}=\emptyset$ if $i\not=j$ and $\cup A_{i}=\Omega $), we have

(\textbf{A4'}) $\sum_{i=1}^{N}1_{A_{i}}\mathbb{E}_{t,T}^{g}[X_{i}]=\mathbb{E}%
_{t,T}^{g}[\sum_{i=1}^{N}1_{A_{i}}X_{i}]$.
\end{proposition}

\begin{lemma}[See \protect \cite{elkarouiPQ1997} or Proposition 2.2 in
\protect \cite{briandCHMP2000}]
\label{lem 4.1} Let $g$ satisfy (\ref{eqn2.2}) and $m=1$, and let $X\in
L^{2}(\mathcal{F}_{T})$. Then the solution $(Y_{t},Z_{t})_{t\in \lbrack
0,T]} $ of BSDE (\ref{eqn2.1}) satisfies
\begin{eqnarray*}
&&E\left[ \sup_{s\in \lbrack t,T]}(e^{\beta s}\left \vert Y_{s}\right \vert
^{2})+\int_{t}^{T}e^{\beta s}\left \vert Z_{s}\right \vert ^{2}\,ds|\mathcal{%
F}_{t}\right] \\
&\leq &KE\left[ e^{\beta T}\left \vert X\right \vert
^{2}+(\int_{t}^{T}e^{(\beta /2)s}\left \vert g(s,0,0)\right \vert \,ds)^{2}|%
\mathcal{F}_{t}\right] .
\end{eqnarray*}%
where $\beta =2(\mu +\mu ^{2})$ and $K$ is a positive constant only
depending on $\mu $.
\end{lemma}

\begin{remark}
The above lemma implies that $\mathbb{E}^g[\cdot]$ is continuous in $L^2$.
\end{remark}

The decomposition theorem of $\mathbb{E}^g$-supermartingale obtained in \cite%
{peng1999} (see also \cite{peng2004}) will play an important role in this
paper.

\begin{proposition}[Decomposition theorem of $\mathbb{E}^g$-supermartingale]

\label{prop4.1} We assume that $g$ satisfies (\ref{eqn2.2}) and $m=1$. Let $%
Y\in \mathcal{D}_\mathcal{F}^2(0,T)$ be a $g$-supermartingale,
namely, for each $0\le s\le t\le T$,
\begin{equation*}
\mathbb{E}_{s,t}^g[Y_t]\le Y_s.
\end{equation*}
Then there exists a unique $\mathcal{F}_t$-adapted increasing and RCLL
process $A\in \mathcal{D}_\mathcal{F}^2(0,T)$ (thus predictable) with $A_0=0$%
, such that, $Y$ is the solution of the following BSDE:
\begin{equation*}
Y_t=Y_T+(A_T-A_t)+\int_t^T g(t,Y_s,Z_s)-\int_t^T Z_s\,dW_s,\quad t\in[0,T].
\end{equation*}
\end{proposition}

\section{$g$-Convexity for $C^{2}$-functions}

\label{sec3} To begin with we give the notion of $g$-convexity.

\begin{definition}
\label{defng-convexity} For a given $g$-expectation $\mathbb{E}^{g}[\cdot ]$%
, a function $h:\mathbb{R}\rightarrow \mathbb{R}$ is said to be $g$-convex
(resp. $g$-concave) if for each $X\in L^{2}(\mathcal{F}_{T})$ such that $%
h(X)\in L^{2}(\mathcal{F}_{T})$, one has
\begin{equation}
h(\mathbb{E}_{t,T}^{g}[X])\leq \mathbb{E}_{t,T}^{g}[h(X)],\  \text{(resp. }h(%
\mathbb{E}_{t,T}^{g}[X])\leq \mathbb{E}_{t,T}^{g}[h(X)]\text{)}\ P\text{-}%
a.s.,\ t\in \lbrack 0,T].  \tag{J}  \label{eqnJ}
\end{equation}%
$h$ is called $g$-affine if it is both $g$-convex and $g$-concave.
\end{definition}

Clearly, for each $g$ the function $h(y)=y$ is $g$-affine. Throughout this
paper, we only consider the case where $h$ is continuous. In the case when $h
$ is a $C^{2}$-function, we have the following result. For notational
convenience, we denote%
\begin{equation*}
\mathcal{L}_{g}^{t,y,z}\varphi :=\frac{1}{2}\varphi
_{yy}(y)|z|^{2}+g(t,\varphi (y),\varphi _{y}(y)z)-\varphi _{y}(y)g(t,y,z),\
\  \varphi \in C^{2}(\mathbb{R}).
\end{equation*}

\begin{theorem}
\label{thm-g-convexity} Let $g$ satisfy (\ref{eqn2.2}) and let $h\in C^{2}(%
\mathbb{R})$. Then the following two statements are equivalent:

(i) $h$ is $g$-convex (resp. $g$-concave);

(ii) For each $y\in \mathbb{R},\ z\in \mathbb{R}^{d}$,%
\begin{equation}
\mathcal{L}_{g}^{t,y,z}h \geq 0\text{ (resp. }\leq 0),\  \ dP\times dt\text{%
--a.s.}\   \label{g-convexity-differentiable}
\end{equation}
\end{theorem}

\begin{remark}
If we assume furthermore $g(t,y,0)\equiv0$ ((\ref{eqn2.3})-b), then we can
define
\begin{equation*}
\mathbb{E}^{g}[X|\mathcal{F}_{t}]=\mathbb{E}_{t,T}^{g}[X].
\end{equation*}
Thus the Jensen's inequality becomes%
\begin{equation*}
\mathbb{E}^{g}[h(X)|\mathcal{F}_{t}]\geq h(\mathbb{E}^{g}[X|\mathcal{F}%
_{t}]).
\end{equation*}
In particular, when $t=0$, $\mathbb{E}^{g}[h(X)]\geq h(\mathbb{E}^{g}[X]).$
\end{remark}

Before the proof of Theorem \ref{thm-g-convexity}, we prove the following
lemma.

\begin{lemma}
\label{prop-g-conv-in-conv} Assume that $g$ satisfies (\ref{eqn2.2}) and an $%
C^2$-function $h$ satisfies (\ref{g-convexity-differentiable}). Then $h$ is
convex in the usual sense.
\end{lemma}

\begin{proof}
For each $y_{0}\in \mathbb{R}$, one has,
\begin{align*}
& 0\leq \frac{1}{2}h^{\prime \prime }(y_{0})\left \vert z\right \vert
^{2}+g(t,h(y_{0}),h^{\prime }(y_{0})z)-h^{\prime }(y_{0})g(t,y_{0},z) \\
& =\frac{1}{2}h^{\prime \prime }(y_{0})\left \vert z\right \vert
^{2}+(g(t,h(y_{0}),h^{\prime }(y_{0})z)-g(t,0,0)) \\
& \quad +h^{\prime }(y_{0})(g(t,0,0)-g(t,y_{0},z))+g(t,0,0)-h^{\prime
}(y_{0})(g(t,0,0) \\
& \leq \frac{1}{2}h^{\prime \prime }(y_{0})\left \vert z\right \vert
^{2}+2C_{1}\left \vert h^{\prime }(y_{0})\right \vert \left \vert z\right
\vert +C_{1}+C_{1}\left \vert g(t,0,0)\right \vert
\end{align*}%
where $C_{1}$ only depends on $\mu $ and $y_{0}$. Thus
\begin{equation*}
0\leq \frac{1}{2}h^{\prime \prime }(y_{0})\left \vert z\right \vert
^{2}T+2C_{2}\left \vert h^{\prime }(y_{0})\right \vert \left \vert z\right
\vert T+C_{2},\  \  \  \forall z\in \mathbb{R}^{d}.
\end{equation*}%
where $C_{2}$ only depends on $C_{1}$ and $M=E\left[ \int_{0}^{T}\left \vert
g(t,0,0)\right \vert ^{2}\,dt\right] $. Thus $h^{\prime \prime }(y_{0})$
must be non-negative.
\end{proof}

\begin{lemma}
\label{bound}Let $h\in C(\mathbb{R})$ be a convex function. If for each $%
X\in L^{\infty }(\mathcal{F}_{T})$ we have
\begin{equation}
\mathbb{E}_{t,T}^{g}[h(X)]\geq h(\mathbb{E}_{t,T}^{g}[X]),\  \ a.s.,\  \
\forall t\in \lbrack 0,T],  \label{Jens}
\end{equation}%
then this relation also holds for each $X\in L^{2}(\mathcal{F}_{T})$ such
that $h(X)\in L^{2}(\mathcal{F}_{T}).$
\end{lemma}

\begin{proof}
We need to consider two cases: (a) $h$ is a monotone function; (b) there
exists a $\bar{y}\in \mathbb{R}$ such that $h(y)\geq h(\bar{y})$. For case
(a), we have
\begin{equation*}
\mathbb{E}_{t,T}^{g}[h((-n)\vee X\wedge m)]\geq h(\mathbb{E}%
_{t,T}^{g}[(-n)\vee X\wedge m]),\ a.s.\ m,n=1,2,\cdots .
\end{equation*}%
Since, for each fixed $n$, the sequence $\{h((-n)\vee X\wedge
m)\}_{m=1}^{\infty }$ (resp. $\{(-n)\vee X\wedge m\}_{m=1}^{\infty }$)
monotonically converges to $h((-n)\vee X)$ (resp. $(-n)\vee X$) in $L^{2}(%
\mathcal{F}_{T})$ as $m\rightarrow \infty $. We then can pass limit on the
both sides of the above inequality and obtain
\begin{equation*}
\mathbb{E}_{t,T}^{g}[h((-n)\vee X)]\geq h(\mathbb{E}_{t,T}^{g}[(-n)\vee
X]),\ a.s.\ n=1,2,\cdots .
\end{equation*}%
Similarly, when $n\uparrow \infty $, the sequence $\{h((-n)\vee
X)\}_{n=1}^{\infty }$ (resp. $\{(-n)\vee X\}_{n=1}^{\infty }$) monotonically
converges to $h(X)$ (resp. $X$) in $L^{2}(\mathcal{F}_{T})$. Thus we can
pass limit on the both sides of the above inequality and obtain (\ref{Jens}%
). For case (b), we observe that then $h$ increases on $[\bar{y},\infty )$
and decreases on $(-\infty ,\bar{y}]$ thus, as $m\rightarrow \infty $, $%
h((-m+\bar{y})\vee X\wedge (m+\bar{y}))$ increasingly converges to $h(X)$ in
$L^{2}(\mathcal{F}_{T})$. We then pass limit on both sides of%
\begin{eqnarray*}
\mathbb{E}_{t,T}^{g}[h((-m+\bar{y})\vee X\wedge (m+\bar{y}))] &\geq &h(%
\mathbb{E}_{t,T}^{g}[(-m+\bar{y})\vee X\wedge (m+\bar{y})]),\  \\
a.s.\ m &=&1,2,\cdots .
\end{eqnarray*}%
and thus obtain (\ref{Jens}).
\end{proof}

\begin{proof}[Proof of Theorem \protect \ref{thm-g-convexity}]
(ii) $\Longrightarrow $(i): We first consider the case where $X$ is bounded.
The corresponding solution $Y_{\cdot }$ of $(g,T,X)$ is also bounded since $%
X $ and $g(\cdot ,0,0)$ are bounded. We now apply It\^{o}'s formula to $%
h(Y_{t})$:
\begin{align*}
-dh(Y_{t})& =[-\frac{1}{2}h^{\prime \prime }(Y_{t})|Z_{t}|^{2}+h^{\prime
}(Y_{t})g(t,Y_{t},Z_{t})]dt-h^{\prime }(Y_{t})Z_{t}dW_{t} \\
& =[g(t,h(Y_{t}),h^{\prime }(Y_{t})Z_{t})+\psi _{t}]dt-h^{\prime
}(Y_{t})Z_{t}dW_{t},
\end{align*}%
where
\begin{equation*}
\psi _{t}=-\frac{1}{2}h^{\prime \prime
}(Y_{t})|Z_{t}|^{2}-g(t,h(Y_{t}),h^{\prime }(Y_{t})Z_{t})+h^{\prime
}(Y_{t})g(t,Y_{t},Z_{t}).
\end{equation*}%
From (\ref{g-convexity-differentiable}), it follows that $\psi
_{t}\leq 0$ and thus $h(Y_{\cdot })$ is a $g$-subsolution. By
comparison theorem of BSDE
it follows that%
\begin{equation*}
\mathbb{E}_{t,T}^{g}[h(X)]=\mathbb{E}_{t,T}^{g}[h(Y_{T})]\geq h(Y_{t})=h(%
\mathbb{E}_{t,T}^{g}[X]),\ a.s.,\  \  \forall X\in L^{\infty }(\mathcal{F}%
_{T}).
\end{equation*}%
On the other hand, from Lemma \ref{prop-g-conv-in-conv} $h$ is
convex. This with Lemma \ref{bound} yields (i).

(i) $\Longrightarrow $(ii): We only give a proof for the situation where $%
g(\cdot ,y,z)$ is a continuous process on $[0,T]$. For each fixed $%
(t,y,z)\in \lbrack 0,T]\times \mathbb{R\times R}^{d}$, we consider the
following SDE
\begin{equation*}
-dY_{s}^{t,y,z}=1_{[t,T]}(s)[g(s,Y_{s}^{t,y,z},z)ds-zdW_{s}],\ Y_{0}=y.\
\end{equation*}%
We apply It\^{o}'s formula on $[t,T]$:
\begin{equation*}
-dh(Y_{s}^{t,y,z})=[-\frac{1}{2}h^{\prime \prime
}(Y_{s}^{t,y,z})|z|^{2}+h^{\prime
}(Y_{s}^{t,y,z})g(s,Y_{s}^{t,y,z},z)]ds-h^{\prime }(Y_{s}^{t,y,z})zdW_{s}
\end{equation*}%
For a large number $m>0$, let $\tau _{m}=\inf \{s\geq
t:|Y_{s}^{t,y,z}-y|=m\} $. It is that $Y_{\cdot }^{t,y,z}$ is bounded on $%
[t,\tau _{m}]$. By (i),
\begin{equation*}
\mathbb{E}_{r\wedge \tau _{m},s\wedge \tau _{m}}^{g}[h(Y_{s\wedge \tau _{m}}^{t,y,z})]\geq h(%
\mathbb{E}_{r\wedge \tau _{m},s\wedge \tau _{m}}^{g}[Y_{s\wedge \tau
_{m}}^{t,y,z}])=h(Y_{r\wedge \tau _{m}}^{t,y,z}),\ P-a.s.,\forall r\in[%
t,\tau _{m}].
\end{equation*}%
That is, $h(Y_{s}^{t,y,z})$ is a $g$--submartingale on $[t,\tau _{m}]$. By
the decomposition theorem of $g$-submartingale (Proposition \ref{prop4.1}),
it follows that there exist an increasing process $(A_{s})_{s\geq t}$ such
that
\begin{equation*}
h(Y_{s\wedge \tau _{m}}^{t,y,z})=h(y)-\int_{t}^{s\wedge \tau
_{m}}g(r,h(Y_{r}^{t,y,z}),Z_{r})\,dr+A_{s\wedge \tau
_{m}}+\int_{t}^{s\wedge \tau _{m}}Z_{r}\,dW_{r}.
\end{equation*}%
This with
\begin{align*}
h(Y_{s\wedge \tau _{m}}^{t,y,z})& =h(y)-\int_{t}^{s\wedge \tau
_{m}}[\frac{1}{2}h^{\prime \prime }(Y_{r}^{t,y,z})|z|^{2}-h^{\prime
}(Y_{r}^{t,y,z})g(r,Y_{r}^{t,y,z},z)]dr \\
& \quad +\int_{t}^{s\wedge \tau _{m}}h^{\prime
}(Y_{r}^{t,y,z})z\,dW_{r}
\end{align*}%
yields $Z_{s}\equiv h^{\prime }(Y_{s}^{t,y,z})z$ and
\begin{equation*}
-\frac{1}{2}h^{\prime \prime }(Y_{s}^{t,y,z})|z|^{2}+h^{\prime
}(Y_{s}^{t,y,z})g(s,Y_{s}^{t,y,z},z)\leq g(s,h(Y_{s}^{t,y,z}),h^{\prime
}(Y_{s}^{t,y,z})z)
\end{equation*}%
on $[t,\tau _{m}]$. Since $g(\cdot ,y,z)$ is a continuous process (otherwise
a technique in the proof of Theorem 8.1 in \cite{peng2005b} is needed), as $%
s=t$, we can obtain (\ref{g-convexity-differentiable}). The proof is
complete.
\end{proof}

\begin{example}
For the case $g=\left \langle \xi _{t},z\right \rangle $, the $g$%
-expectation corresponds to the classical linear expectation (Girsanov
transformation). Theorem \ref{thm-g-convexity}, becomes the classical
results: $h\in C^{2}(\mathbb{R})$ is $g$-convex if and only if $h^{\prime
\prime }(y)\geq 0$.
\end{example}

\begin{remark}
\label{rem-g-concave-is-quasiconcave} A $g$-concave function is concave in
the usual sense. Its proof is similar to that of Corollary \ref%
{thm-g-convex-convex}.
\end{remark}

From Theorem \ref{thm-g-convexity} we can also derive the following result
of \cite{chenKJ2003} and its improved version \cite{hu2005}.

\begin{proposition}
\label{corhu} Let $g$ satisfy (\ref{eqn2.2}). Then the following two
statements are equivalent:

(i) For each convex function $h$, and each $X\in L^{2}(\mathcal{F}_{T})$
such that $h(X)\in L^{2}(\mathcal{F}_{T})$,
\begin{equation*}
\mathbb{E}_{t,T}^{g}[h(X)]\geq h(\mathbb{E}_{t,T}^{g}[X]),\quad \forall
0\leq t\leq T;
\end{equation*}

(ii) $g$ is independent of $y$, and is super-homogeneous in $z$, i.e., for
any $\lambda \in \mathbb{R}$, $g(t,\lambda z)\geq \lambda g(t,z)$.
\end{proposition}

\begin{proof}
(ii) $\Rightarrow $ (i): In the case when $h\in C^{2}$, this can be proved
by (\ref{g-convexity-differentiable}). For general situation we can apply
the same technique in the proof of (\ref{g-convexity-differentiable}) (see
\cite{chenKJ2003}).

(i) $\Rightarrow $ (ii): For each given $a,b\in \mathbb{R}$, take $h(x)=ax+b$%
. Obviously it is a convex function and in $C^{2}(\mathbb{R})$. Thus the
inequality (\ref{g-convexity-differentiable}) yields $g(t,ay+b,az)-ag(t,y,z)%
\geq 0,\  \ dP\times dt\text{--a.s.}$ Since $a,b$ can be chosen arbitrarily, $%
g$ must be independent of $y$ and super-homogeneous in $z$.
\end{proof}

\begin{corollary}
\label{cor-g-conv-ind-z} Let $g$ satisfy (\ref{eqn2.2}) and be independent
of $z$, and $h\in C^{2}(\mathbb{R})$. Then the following two statements
equivalent:

(i) $h$ is $g$-convex;

(ii) $h$ is convex ($h^{\prime \prime }(y)\geq 0$ for each $y$) and
satisfies
\begin{equation*}
\forall y,\quad g(t,h(y))-h^{\prime }(y)g(t,y)\geq 0,\ dP\times dt\text{%
--a.s.}
\end{equation*}
\end{corollary}

\begin{corollary}
\label{cor-g-conv-ind-y} Let $g$ satisfy (\ref{eqn2.2}) and be independent
of $y$, and let $h\in C^{2}(\mathbb{R})$ be $g$-convex. Moreover if there
exist a set $\Gamma \in \Omega \times \lbrack 0,T]$ with positive measure,
in which $g(t,0)>0$ (resp. $g(t,0)<0$), then $h^{\prime }(y)\leq 1\ $(resp. $%
h^{\prime }(y)\geq 1$)$.$
\end{corollary}

A simple and fundamentally important result in stochastic analysis is that,
for each martingale $X$ and for each convex function $h$ such that $h(X)\in
L^{1}$, the process $h(X)$ is a submartingale. For $g$-expectation, we have:

\begin{theorem}
\label{thm-g-submartingale} If $(Y_{t})_{t\in \lbrack 0,T]}$ is a $g$%
-martingale, and $h$ is a $g$-convex function (resp. $g$-concave function, $%
g $-affine function), then $(h(Y_{t}))_{t\in \lbrack 0,T]}$ is a $g$%
-submartingale (resp. $g$-supermartingale, $g$-martingale) provided $%
h(Y_{t})\in L^{2}(\mathcal{F}_{t})$, $t\in \lbrack 0,T]$.
\end{theorem}

\begin{proof}
Let $Y_{t}$ be a $g$-martingale and $h$ a $g$-convex function, then
\begin{equation*}
\mathbb{E}_{s,t}[h(Y_{t})|\mathcal{F}_{s}]\geq h(\mathbb{E}_{s,t}[Y_{t}|%
\mathcal{F}_{s}])=h(Y_{s}),
\end{equation*}%
for any $0\leq s<t\leq T$, as required. The proofs of other cases are
similar.
\end{proof}

Moreover its inverse also holds, namely,

\begin{theorem}
\label{thm-g-submartingale-inverse} Let $g$ satisfy (\ref{eqn2.2}).
If for each $g$-martingale $(Y_{t})_{t\in \lbrack 0,T]}$,
$(h(Y_{t}))_{t\in \lbrack 0,T]}$ is a $g$-submartingale (resp.
$g$-submartingale, $g$-martingale), then $h$ is a $g$-convex (resp.
$g$-concave ,$g$-affine) function.
\end{theorem}

\begin{proof}
We only prove the case of $g$-submartingale. Since, for each $X\in L^{2}(%
\mathcal{F}_{T})$, $(\mathbb{E}_{t,T}^{g}[X])_{t\in \lbrack 0,T]}$ is a $g$%
-martingale. Define $\bar{Y}_{t}=h(\mathbb{E}_{t,T}^{g}[X])_{t\in \lbrack
0,T]}$, we have, for $0\leq s<t\leq T$,
\begin{equation*}
\mathbb{E}_{s,T}^{g}[h(\mathbb{E}_{t,T}^{g}[X])]=\mathbb{E}_{s,t}^{g}[h(%
\mathbb{E}_{t,T}^{g}[X])]=\mathbb{E}_{s,t}^{g}[\bar{Y}_{t}]\geq \bar{Y}%
_{s}=h(\mathbb{E}_{s,T}^{g}[X])
\end{equation*}%
In particular, as $t=T$, it follows that $\mathbb{E}_{s,T}^{g}[h(X)]\geq h(%
\mathbb{E}_{s,T}^{g}[X])$ for $s\in \lbrack 0,T]$. Thus $h$ is a
$g$-convex function.
\end{proof}

\section{$g$-Convexity for Continuous Functions}

\label{sec4}

In this section we consider $g$-convex functions $h\in C(\mathbb{R})$, i.e.,
without the $C^{2}$-assumption.
%We only consider the case where $g$ is a
%deterministic function i.e., $g:[0,T]\times \mathbb{R}\times \mathbb{R}%
%^{d}\mapsto \mathbb{R}$. We also assume the Lipschitz condition:
%\begin{equation}
%\begin{cases}
%& \text{There exists a constant }\mu >0\text{, such that}, \\
%& \qquad \left \vert g(t,y,z)-g(t,\bar{y},\bar{z})\right \vert \leq \mu (|y-%
%\bar{y}|+\left \vert z-\bar{z}\right \vert );\  \  \  \forall z,\bar{z}\in
%\mathbb{R}^{d}; \\
%& \  \ (g(\cdot ,y,z))_{t\in \lbrack 0,T]}\text{ has continuous path and }%
%\max_{t\in \lbrack 0,T]}|g(t,0,0)|<\infty .%
%\end{cases}
%\label{eqn5.2}
%\end{equation}

We now recall the definition of viscosity subsolutions.

\begin{definition}
\label{defn-viscosity-solution} Let $g$ satisfy (\ref{eqn2.2}) and independent of $\oo$. A continuous function $u:\mathbb{R}\to%
\mathbb{R}$ is called a viscosity subsolution of
$\mathcal{L}_{g}^{t,y,z}u=0$
if, for any $\varphi \in C^2(\mathbb{R})$, and $x\in \mathbb{R}$ such that $%
u-\varphi$ attains local maximum at $x$, one has for each $(t,z)\in[0,T]%
\times \mathbb{R}$,
\begin{equation*}
\mathcal{L}_{g}^{t,x,z}\varphi=\frac{1}{2}\varphi^{\prime \prime }(x)\left
\vert z\right \vert^2+g(t,u(x),\varphi^{\prime }(x)z)-\varphi^{\prime
}(x)g(t,x,z)\ge 0
\end{equation*}
\end{definition}

\begin{theorem}
\label{thm-viscosity-is-g-convex} Let $h\in C(\mathbb{R})$ be of
polynomial growth. Moreover let us assume that $g$ satisfies
(\ref{eqn2.2}) and is independent of $\oo$. The the following
conditions are equivalent:

(i) $h$ is a viscosity subsolution of $\mathcal{L}_{g}^{t,y,z}h=0$;

(ii) $h$ is $g$-convex.
\end{theorem}

\begin{remark}
For more basic definitions, results and related literature on
viscosity solutions of PDE, we refer to Crandall, Ishii and Lions
\cite{crandallIL1992}.
\end{remark}

For proving this theorem, we need the following lemma.

\begin{lemma}
\label{lem-sub-is-convex} If $g$ satisfies (\ref{eqn2.2}) and $h$ is a
continuous viscosity subsolution of $\mathcal{L}_{g}^{t,y,z}h=0$ for each $%
(t,z)\in \lbrack 0,T]\times \mathbb{R}^{d}$, then $h$ is convex in the usual
sense.
\end{lemma}

\begin{proof}
If on the contrary $h$ is not convex, then there are constants $-\infty
<a<b<\infty $ such that the relation $\psi \geq h$ fails on $[a,b]$, where
\begin{equation*}
\psi (x):=\frac{h(b)(x-a)}{b-a}+\frac{h(a)(b-x)}{b-a}.
\end{equation*}%
We set $h_{\delta }(x):=\psi (x)-\delta (x-a)(x-b)$ and%
\begin{equation*}
\delta _{0}=\inf \{ \delta >0:h_{\delta }(x)\geq h(x),\  \  \forall x\in
\lbrack a,b]\}.
\end{equation*}%
It is easy to check that $\delta _{0}>0$, $h_{\delta _{0}}\geq h$ on $[a,b]$
and there exists $\bar{x}\in (a,b)$ such that $h_{\delta _{0}}(\bar{x})=h(%
\bar{x})$. But since for each $z\in \mathbb{R}^{d}$, $h$ is a viscosity
subsolution of $\mathcal{L}_{g}^{t,z}h=0$, and $h_{\delta _{0}}-h$ attaints
minimum at $\bar{x}$, we have
\begin{eqnarray*}
0 &\leq &\mathcal{L}_{g}^{t,z}h_{\delta _{0}}(\bar{x})=-\delta
_{0}|z|^{2}+g(t,[\frac{h(b)-h(a)}{b-a}-\delta _{0}(2\bar{x}-b+a)]z) \\
&&+[\frac{h(b)-h(a)}{b-a}-\delta _{0}(2\bar{x}-b+a)]g(t,z).\  \  \
\end{eqnarray*}%
Since $g$ is Lipschitz in $z$, there exists a positive constant $C$
independent of $z$, such that
\begin{equation*}
-\delta _{0}|z|^{2}+C_{1}|z|+C_{2}\geq 0,\  \  \forall z\in \mathbb{R}^{d}.
\end{equation*}%
This contradicts to $\delta _{0}>0$. Thus $h$ must be convex.
\end{proof}

Combining this Lemma with Theorem \ref{thm-viscosity-is-g-convex} we
immediately have a more explicit characterization for a continuous $g$%
-convex function:

\begin{corollary}
We assume the same conditions as in the above theorem. Then the following
condition is equivalent:

(i) $h$ is convex and for each $y$ such that $h^{\prime \prime }(y)$ exists,
$\mathcal{L}_{g}^{t,y,z}h(y)\geq 0$;

(ii) $h$ is $g$-convex.
\end{corollary}

\begin{proof}
If $h$ is a viscosity subsolution of $\mathcal{L}_{g}^{t,y,z}h=0$ then $h$
is convex. On the other hand, by Alvarez, Lasry and Lions \cite%
{alvarezLL1997}, if $h$ is convex and for each $y$ such that $h^{\prime
\prime }(y)$ exists, one has $\mathcal{L}_{g}^{t,y,z}h(y)\geq 0$, then $h$
is a viscosity subsolution of $\mathcal{L}_{g}^{t,y,z}h=0$.
\end{proof}

\, \,

\begin{proof}[The proof of Theorem \protect \ref{thm-viscosity-is-g-convex}(i)%
$\Longrightarrow $(ii)]
Given $(t,x,z)\in \lbrack 0,T]\times \mathbb{R}\times \mathbb{R}^{d}$, we
consider the following SDE
\begin{equation*}
dX_{s}^{t,x;z}=-g(s,X_{s}^{t,x;z},z)ds+zdW_{s},\ s\in (t,T],\  \
X_{s}^{t,x;z}=x,\  \ s\in \lbrack 0,t].
\end{equation*}%
It is clear that $X^{t,x;z}$ is also a $g$-martingale on $[0,T]$: In
particular $\mathbb{E}_{s,T}^{g}[X_{T}^{t,x;z}]=X_{s}^{t,x;z}$, and $\mathbb{%
E}_{t,T}^{g}[X_{T}^{t,x;z}]=\mathbb{E}^{g}[X_{T}^{t,x;z}]=x$. On the other
hand, by nonlinear Feynman-Kac formula, the function $u(t,x):=\mathbb{E}%
^{g}[h(X_{T}^{t,x;z})]$ defined on $[0,T]\times \mathbb{R}$ is the viscosity
solution of the parabolic PDE%
\begin{equation*}
\partial _{t}u+\frac{1}{2}\partial _{xx}u(t,x)|z|^{2}-\partial
_{x}ug(t,x,z)+g(t,u,z\partial _{x}u)=0,\  \ u|_{t=T}=h(x).
\end{equation*}%
But the function defined by $v(t,x):=h(x)$ is a viscosity subsolution of $%
\partial _{t}v+\mathcal{L}_{g}^{t,z}v=0$ with terminal condition $v|_{t=T}=h$%
. It follows from the maximum principle of viscosity solution that
\begin{equation*}
u(t,x)\geq h(x),\  \  \forall (t,x)\in \lbrack 0,T]\times \mathbb{R}.
\end{equation*}%
Or
\begin{align*}
\mathbb{E}^{g}[h(X_{T}^{t,x;z})|\mathcal{F}_{t}]& =\mathbb{E}%
_{t,T}^{g}[h(X_{T}^{t,x;z})] \\
& \geq h(x)=h(\mathbb{E}^{g}[X_{T}^{t,x;z}])=h(\mathbb{E}%
_{t,T}^{g}[X_{T}^{t,x;z}|\mathcal{F}_{t}]).
\end{align*}%
We now apply a technique initialed in \cite[pp.107, Theorem 4.6;
Peng1995:Xiangfan Summer School]{peng1997b}: Let $\{A_{i}\}_{i=1}^{N}$ be an
$\mathcal{F}_{t}$-measurable partition of $\Omega $; \ and $z_{i}\in \mathbb{%
R}^{n}$, $x_{i}\in \mathbb{R}$, $i=1,\cdots ,N$ be given. We set $\eta
=\sum_{i=1}^{N}\mathbf{1}_{A_{i}}z_{i}$, $\zeta =\sum_{i=1}^{N}\mathbf{1}%
_{A_{i}}x_{i}$. It is easy to check that $\sum_{i=1}^{N}\mathbf{1}%
_{A_{i}}X_{s}^{t,x_{i};z_{i}}=X_{T}^{t,\zeta ;\eta }$.
\begin{eqnarray*}
\sum_{i=1}^{N}\mathbf{1}_{A_{i}}\mathbb{E}^{g}[h(X_{T}^{t,x_{i};z_{i}})|%
\mathcal{F}_{t}] &=&\mathbb{E}^{g}[\sum_{i=1}^{N}\mathbf{1}%
_{A_{i}}h(X_{T}^{t,x_{i};z_{i}})|\mathcal{F}_{t}] \\
&=&\mathbb{E}^{g}[h(\sum_{i=1}^{N}\mathbf{1}_{A_{i}}X_{T}^{t,x_{i};z_{i}})|%
\mathcal{F}_{t}] \\
&=&\mathbb{E}^{g}[h(X_{T}^{t,\zeta ;\eta })|\mathcal{F}_{t}]
\end{eqnarray*}%
Thus
\begin{eqnarray*}
\mathbb{E}^{g}[h(X_{T}^{t,\zeta ;\eta })|\mathcal{F}_{t}] &=&\sum_{i=1}^{N}%
\mathbf{1}_{A_{i}}\mathbb{E}^{g}[h(X_{T}^{t,x_{i};z_{i}})|\mathcal{F}%
_{t}]\geq \sum_{i=1}^{N}\mathbf{1}_{A_{i}}\mathbb{E}%
^{g}[h(X_{T}^{t,x_{i};z_{i}})|\mathcal{F}_{t}] \\
&\geq &\sum_{i=1}^{N}\mathbf{1}_{A_{i}}h(\mathbb{E}%
^{g}[X_{T}^{t,x_{i};z_{i}}|\mathcal{F}_{t}])=h(\sum_{i=1}^{N}\mathbf{1}%
_{A_{i}}\mathbb{E}^{g}[X_{T}^{t,x_{i};z_{i}}|\mathcal{F}_{t}]) \\
&=&h(\mathbb{E}^{g}[X_{T}^{t,\zeta ;\eta }|\mathcal{F}_{t}])=h(\zeta ).
\end{eqnarray*}%
\begin{equation*}
\mathbb{E}_{t,T}^{g}[h(X_{T}^{t,\zeta ;\eta })]=\sum_{i=1}^{N}1_{A_{i}}%
\mathbb{E}_{t,T}^{g}[h(X_{T}^{t,x_{i};z_{i}})]\geq
\sum_{i=1}^{N}1_{A_{i}}h(x_{i})=h(\zeta ),
\end{equation*}%
In other words, for bounded $\mathcal{F}_{t}$-measurable simple functions $%
\zeta ,\eta $,
\begin{equation}
\mathbb{E}_{t,T}^{g}[h(\zeta -\int_{t}^{T}g(s,X_{s}^{t,\zeta ;\eta },\eta
)ds+\int_{t}^{T}\eta dW_{s})]\geq h(\zeta )  \label{eqninequality}
\end{equation}%
It follows that for any bounded $\mathcal{F}_{t}$-measurable random
variables $\zeta ,\eta $, we also have (\ref{eqninequality}). Moreover, for
any bounded $\mathcal{F}_{t}$-adapted process $\eta $ and bounded $\mathcal{F%
}_{t}$-measurable random variables $\zeta $, we have
\begin{equation}
\mathbb{E}_{t,T}^{g}[h(\zeta -\int_{t}^{T}g(s,X_{s}^{t,\zeta ;\eta },\eta
_{s})ds+\int_{t}^{T}\eta _{s}dW_{s})]\geq h(\zeta )  \label{eqninequality1}
\end{equation}%
Indeed we note that $\zeta -\int_{t}^{T}g(s,X_{s}^{t,\zeta ;\eta },\eta
_{s})ds+\int_{t}^{T}\eta _{s}dW_{s}\in L^{2m+1}(\mathcal{F}_{T})$, this with
the polynomial growth of $h$ %Lemma \ref{lem-Lp}, Lemma \ref{lem-Xi}
and the continuity of $\mathbb{E}^{g}[\cdot ]$ yields (\ref{eqninequality1}).

Now for any given bounded $\mathcal{F}_{T}$-measurable $X$, let $%
(Y_{s},Z_{s})_{s\in \lbrack 0,T]}$ be the solution of the BSDE
\begin{equation*}
Y_{t}=X+\int_{t}^{T}g(s,Y_{s},Z_{s})ds-\int_{t}^{T}Z_{s}dW_{s}
\end{equation*}%
Let $h$ be a function with polynomial growth $|h(x)|\leq C(1+|x|^{m})$. We
note that $\sup_{t\in \lbrack 0,T]}|Y_{t}(\omega )|\in L^{\infty }(\mathcal{F%
}_{T})$ and,
\begin{equation*}
E[\left( \int_{0}^{T}\left \vert Z_{s}\right \vert ^{2}ds\right) ^{\frac{2m+1%
}{2}}]<\infty .
\end{equation*}%
We can find a sequence of $\mathcal{F}_{t}$-measurable simple functions $\{
\zeta ^{i}\}_{i=1}^{\infty }$ that converges to $Y_{t}$ in $L^{\infty }(%
\mathcal{F}_{t})$ and a sequence of $\mathcal{F}_{t}$-progressively
measurable simple processes $\{(\eta _{t}^{i})_{t\in \lbrack
0,T]}\}_{i=1}^{\infty }$ such that%
\begin{equation*}
\lim_{i\rightarrow \infty }E[\left( \int_{0}^{T}\left \vert Z_{s}-\eta
_{s}^{i}\right \vert ^{2}ds\right) ^{^{\frac{2m+1}{2}}}]=0.
\end{equation*}%
It follows from BDG-inequality that the random variables
\begin{equation*}
X^{i}:=\zeta ^{i}-\int_{t}^{T}g(s,X_{s}^{t,\zeta ^{i};\eta ^{i}},\eta
_{s}^{i})ds+\int_{t}^{T}\eta _{s}^{i}dW_{s}
\end{equation*}
converges in $L^{2m+1}(\mathcal{F}_{T})$ to $X$. Thus $h(X^{i})$ converges
to $h(X)$ in $L^{2}(\mathcal{F}_{T})$. Thus
\begin{equation*}
\mathbb{E}_{t,T}^{g}[h(X)]=\lim_{i\rightarrow \infty }\mathbb{E}%
_{t,T}^{g}[h(X^{i})]\geq \lim_{i\rightarrow \infty }h(\zeta ^{i})=h(Y_{t})=h(%
\mathbb{E}_{t,T}^{g}[X]).
\end{equation*}%
Thus (ii) holds for the case where $X\in L^{\infty }(\mathcal{F}_{T})$. This
with the fact that $h$ is convex and Lemma \ref{bound} it follows that (ii)
holds for all $X\in L^{2}(\mathcal{F}_{T})$ such that $h(X)\in L^{2}(%
\mathcal{F}_{T})$. The proof is complete.
\end{proof}

\vspace{0.2cm}

\begin{proof}[The proof of Theorem \protect \ref{thm-viscosity-is-g-convex}%
(ii)$\Longrightarrow $(i)]
We will apply a technique in \cite[pp.126]{peng1997b}. For a fixed $t,x,z$,
let $\varphi $ be a smooth and polynomial growth function such that $\varphi
\geq h$ and $h(x)=\varphi (x)$. We consider
\begin{equation*}
X_{s}^{t,x;z}=x-\int_{s}^{s}g(r,X_{r}^{t,x;z},z)dr+z(W_{s}-W_{t}),\  \  \ s\in
\lbrack t,t+\delta ].
\end{equation*}%
where $\delta $ is a small positive number such that $t+\delta \leq T$. It
is clear that $X^{t,x;z}$ is a $g$-martingale. Since $h$ is $g$-convex, we
have
\begin{equation*}
\mathbb{E}_{t,t+\delta }^{g}[\varphi (X_{t+\delta }^{t,x;z})]\geq \mathbb{E}%
_{t,t+\delta }^{g}[h(X_{t+\delta }^{t,x;z})]\geq h(\mathbb{E}_{t,t+\delta
}^{g}[X_{t+\delta }^{t,x;z}])=h(x)=\varphi (x).
\end{equation*}%
Or%
\begin{equation*}
Y_{t}-\varphi (x)=\mathbb{E}_{t,t+\delta }^{g}[\varphi (X_{t+\delta
}^{t,x;z})]-\varphi (x)\geq 0
\end{equation*}%
where $(Y,Z)$ solve the BSDE%
\begin{equation*}
\begin{cases}
-dY_{s} & =g(s,Y_{s},Z_{s})ds-Z_{s}dW_{s},\  \ s\in \lbrack t,t+\delta ], \\
Y_{t+\delta } & =\varphi (X_{t+\delta }^{t,x;z}).%
\end{cases}%
\end{equation*}%
We consider
\begin{equation*}
Y_{s}^{1}:=Y_{s}-\varphi (X_{s}^{t,x;z}),\  \  \ Z_{s}^{1}:=Z_{s}-\varphi
_{x}(X_{s}^{t,x;z})z
\end{equation*}%
which is the solution of the BSDE%
\begin{equation*}
\begin{cases}
-dY_{s}^{1} & =[g(s,Y_{s}^{1}+\varphi (X_{s}^{t,x;z}),Z_{s}^{1}+\varphi
_{x}(X_{s}^{t,x;z})z)+\mathcal{\bar{L}}^{s}\varphi
(X_{s}^{t,x;z})]ds-Z_{s}^{1}dW_{s}, \\
Y_{t+\delta }^{1} & =0.%
\end{cases}%
\end{equation*}%
where $s\in \lbrack t,t+\delta ]$ and $\mathcal{\bar{L}}^{s}\varphi (x)=%
\frac{1}{2}\varphi _{xx}(x)|z|^{2}-\varphi _{x}(x)g(s,x,z)$. We can prove
that $E[|Y_{t}^{1}-Y_{t}^{2}|]=o(\delta )$, where $(Y^{2},Z^{2})$ solves
\begin{equation*}
\begin{cases}
-dY_{s}^{2} & =[g(s,Y_{s}^{2}+\varphi (x),Z_{s}^{2}+\varphi _{x}(x)z)+%
\mathcal{\bar{L}}^{s}\varphi (x)]ds-Z_{s}^{2}dW_{s},\  \ s\in \lbrack
t,t+\delta ], \\
Y_{t+\delta }^{2} & =0.%
\end{cases}%
\end{equation*}%
But It is easy to check that $Z^{2}\equiv 0$ and%
\begin{equation*}
\begin{cases}
-dY_{s}^{2} & =[g(s,Y_{s}^{2}+\varphi (x),\varphi _{x}(x)z)+\mathcal{\bar{L}}%
^{s}\varphi (x)]ds, \\
Y_{t+\delta }^{2} & =0.%
\end{cases}%
\end{equation*}%
Thus from the Lipschitz continuity of $g(s,\cdot ,z)$, we have
\begin{align*}
Y_{t}-\varphi (x)& =Y_{t}^{1}=Y_{t}^{2}+o(\delta ) \\
& =\int_{t}^{t+\delta }[g(s,Y_{s}^{2}+\varphi (x),\varphi _{x}(x)z)+\mathcal{%
\bar{L}}^{s}\varphi (x)]ds+o(\delta ) \\
& =\int_{t}^{t+\delta }[g(s,\varphi (x),\varphi _{x}(x)z)+\mathcal{\bar{L}}%
^{s}\varphi (x)]ds+o(\delta )\geq 0
\end{align*}%
From which it follows that $\mathcal{\bar{L}}^{t}\varphi (x)+g(t,\varphi
(x),\varphi _{x}(x)z)=\mathcal{L}_{g}^{t,x,z}\varphi (x)\geq 0$. Thus $h$ is
a viscosity subsolution of $\mathcal{L}_{g}^{t,y,z}u=0$.
\end{proof}

\section{$g$-Convexity and Viability}

\label{sec5} Surprisingly to us, the notion of $g$-convexity has a deep
relation with the notion of viability for BSDE introduced and systematically
studied by Buckdahn, Quincampoix and Rascanu in \cite{buckdahnQR2000}. We
recall the notion and a result about the backward stochastic viability
property.

\begin{definition}[Definition 3 in \protect \cite{buckdahnQR2000}]
\label{defn-viability} Let $K$ be a nonempty closed subset of $\mathbb{R}^m$.

(a) A stochastic process $(Y_{t})_{t\in \lbrack 0,T]}$ is viable in $K$ if
and only if
\begin{equation*}
Y_{t}\in K,\quad P\text{-}a.s.,\  \forall t\in \lbrack 0,T].
\end{equation*}

(b) The closed set $K$ enjoys the backward stochastic viability property,
denoted $g$-\textbf{BSVP}, for (\ref{eqn2.1}) if and only if:

$\forall \tau \in[0,T],\forall X\in L^2(\mathcal{F}_\tau)$ such that $X\in K$
P-a.s., there exists a solution $(Y,Z)$ to BSDE (\ref{eqn2.1}) over the time
interval $[0,\tau]$,
\begin{equation*}
Y_s=X+\int_s^\tau g(r,Y_r,Z_r)\,dr-\int_s^\tau Z_r\,dW_r,\quad s\in[0,\tau]
\end{equation*}
such that $(Y_s)_{s\in[0,\tau]}$ is viable in $K$.
\end{definition}

\begin{lemma}[Theorem 2.4 in \protect \cite{buckdahnQR2000}]
\label{thm2.4} Suppose that $g$ satisfies condition (\ref{eqn2.2}). Let $K$
be a nonempty closed set. If $K$ enjoys $g$-BSVP for (\ref{eqn2.1}), then $K$
is convex.
\end{lemma}

\begin{remark}
In the above lemma, the authors in \cite{buckdahnQR2000} assume that $g$
also satisfies the following conditions: $g(\omega ,\cdot ,y,z)$ is
continuous, as a part of whole assumptions. But in their proof, we can see
that this condition is needless to this lemma, condition (\ref{eqn2.2}) is
enough.
\end{remark}

\begin{theorem}
\label{thm-g-convex-BSVP} Let $g$ satisfy (\ref{eqn2.2}) and $h:\mathbb{R}%
\rightarrow \mathbb{R}$ be a continuous function. Moreover assume that
\begin{equation*}
\bar{g}(t,y^{1},y^{2},z^{1},z^{2})=\left(
\begin{array}{c}
g(t,y^{1},z^{1}) \\
g(t,y^{2},z^{2})%
\end{array}%
\right) .
\end{equation*}%
Then the following statements are equivalent:

(i). $h$ is $g$-convex;

(ii). $\mathbf{epi}(h)$ enjoys $\bar{g}$-BSVP where
\begin{equation*}
\mathbf{epi}(h)=\left \{ (x_{1},x_{2})\in \mathbb{R}^{2};\ h(x_{1})\leq
x_{2}\right \} .
\end{equation*}
\end{theorem}

\begin{proof}
(i)$\Rightarrow $(ii): It is obvious that $\mathbf{epi}(h)$ is a closed set
in $\mathbb{R}^{2}$.

Given any $X=(X_{1},X_{2})^{T}\in \mathbf{epi}(h)$ P-a.s. such that $X\in
L^{2}(\mathcal{F}_{T},\mathbb{R}^{2})$. By the definition of $\mathbf{epi}$,
we have
\begin{equation*}
h(X_{1})\leq X_{2},\ P-a.s.,
\end{equation*}%
which implies by the comparison theorem of BSDE and the $g$-convexity of $h$
that
\begin{equation*}
h(\mathbb{E}_{t,T}^{g}[X_{1}])\leq \mathbb{E}_{t,T}^{g}[h(X_{1})]\leq
\mathbb{E}_{t,T}^{g}[X_{2}],\ P-a.s.
\end{equation*}%
Thus
\begin{equation}
(\mathbb{E}_{t,T}^{g}[X_{1}],\mathbb{E}_{t,T}^{g}[X_{2}])\in \mathbf{epi}%
(h),\ P-a.s.\ t\in \lbrack 0,T].  \label{eqng-convex-is-convex}
\end{equation}

It is clear that $\bar{g}$ satisfies (\ref{eqn2.2}). Moreover $(\mathbb{E}%
_{t,T}^{g}[X_{1}],\mathbb{E}_{t,T}^{g}[X_{2}])_{t\in \lbrack 0,T]}$ is the
unique solution of the following equation
\begin{equation*}
\left(
\begin{array}{c}
Y_{t}^{1} \\
Y_{t}^{2}%
\end{array}%
\right) =\left(
\begin{array}{c}
X_{1} \\
X_{2}%
\end{array}%
\right) +\int_{t}^{T}\left(
\begin{array}{c}
g(s,Y_{s}^{1},Z_{s}^{1}) \\
g(s,Y_{s}^{2},Z_{s}^{2})%
\end{array}%
\right) \,ds-\int_{t}^{T}\left(
\begin{array}{c}
Z_{s}^{1} \\
Z_{s}^{2}%
\end{array}%
\right) \,dW_{s}.
\end{equation*}%
Then (\ref{eqng-convex-is-convex}) implies that $\mathbf{epi}(h)$ enjoys $%
\bar{g}$-BSVP, as required.

(ii)$\Rightarrow $(i): Assume that $\mathbf{epi}(h)$ enjoys $\bar{g}$-BSVP,
i.e., for any $X=(X_{1},X_{2})^{T}\in L^{2}(\mathcal{F}_{T};\mathbb{R}^{2})$
such that $X\in \mathbf{epi}(h)$, we have
\begin{equation*}
(\mathbb{E}_{t,T}^{g}[X_{1}],\mathbb{E}_{t,T}^{g}[X_{2}])\in \mathbf{epi}%
(h),\quad P-a.s.,\ t\in \lbrack 0,T],
\end{equation*}%
and by the definition of $\mathbf{epi}(h)$,
\begin{equation*}
h(\mathbb{E}_{t,T}^{g}[X_{1}])\leq \mathbb{E}_{t,T}^{g}[X_{2}],\quad
P-a.s.,\ t\in \lbrack 0,T].
\end{equation*}%
For any given $X_{1}$ such that $X_{1}\in L^{2}(\mathcal{F}_{T})$, putting $%
X_{2}=h(X_{1})$ yields
\begin{equation*}
h(\mathbb{E}_{t,T}^{g}[X_{1}])\leq \mathbb{E}_{t,T}^{g}[X_{2}]=\mathbb{E}%
_{t,T}^{g}[h(X_{1})],
\end{equation*}%
as required.
\end{proof}

\begin{remark}
In the proof of Theorem \ref{thm-g-convex-BSVP}, we note that we do not need
condition (b) of (\ref{eqn2.2}), $(g(t,0,0))_{t\in \lbrack 0,T]}\in L_{%
\mathcal{F}}^{2}(0,T)$ is enough.
\end{remark}

\begin{corollary}
\label{thm-g-convex-convex} If a continuous functions $h$ is $g$-convex,
then $h$ is convex.
\end{corollary}

\begin{proof}
It is clear that $\mathbf{epi}(h)$ enjoys $\bar{g}$-BSVP. By Theorem 2.4 in
\cite{buckdahnQR2000}, $\mathbf{epi}(h)$ is a convex set, which implies that
$h$ is a convex function.
\end{proof}

Clearly, a $g$-affine function must be affine in the usual sense.
Then we have

\begin{theorem}
\label{thm-g-affinity} Let $g$ satisfy (\ref{eqn2.2}). Then the following
two statements are equivalent:

(i) A function $h$ is $g$-affine;

(ii) $h$ has the form: $h(y)=ay+b$ for some $(a,b)\in \Pi _{g}^{a}$ where
\begin{equation*}
\Pi _{g}^{a}:=\left \{ (a,b);g(t,ay+b,az)=ag(t,y,z),\ dP\times dt\text{-}%
a.s.\right \}
\end{equation*}
\end{theorem}

\section{More Properties of $g$-Convexity}

\label{sec7}

\subsection{Functional operations preserving $g$-convexity}

\label{subset7.1}

It is natural to build up new $g$-convex functions from simpler ones, via
operations preserving $g$-convexity, or even yielding it.

\begin{proposition}
\label{thm-g-convexity-nonsmooth0} Let $g$ satisfy (\ref{eqn2.2}), $\varphi
\in C(\mathbb{R})$. If $\mathcal{D}$ is a nonempty subset of $g$-convex
functions dominated by $\varphi $, then the function
\begin{equation*}
f(y)=\sup \left \{ h(y):h\in \mathcal{D}\right \} .
\end{equation*}%
is $g$-convex.
\end{proposition}

\begin{proof}
It is clear that $f$ is convex. For any given $h\in \mathcal{D}$, Jensen's
inequality for $g$-expectation holds, thus for any $X\in L^{\infty}(\mathcal{F}%
_{T})$, we have
\begin{equation*}
\mathbb{E}_{t,T}^{g}[h(X)]\geq h(\mathbb{E}_{t,T}^{g}[X]).
\end{equation*}%
From the definition of $f$ and comparison theorem of BSDEs, it follows that
\begin{equation*}
\mathbb{E}_{t,T}^{g}[f(X)]\geq \mathbb{E}_{t,T}^{g}[h(X)]\geq h(\mathbb{E}%
_{t,T}^{g}[X]).
\end{equation*}%
This with the arbitrariness of $h$ and Lemma \ref{bound} yields what
is required.
\end{proof}

Clearly, the function $f$ in Theorem \ref{thm-g-convexity-nonsmooth0} may be
only continuous instead of in $C^{2}$ and if $h_{1}$ and $h_{2}$ are $g$%
-convex, then so is $h(y)=h_{1}(y)\vee h_{2}(y)$. In addition, for the case
of $g$-concavity, we also have the same result, in which the "sup" is
replaced by "inf".

The following result is also easy:

\begin{proposition}
\label{prop-gpi-g-equivalence} Let $\varphi \in C^{2}(\mathbb{R})$,
$g$ satisfy (\ref{eqn2.2}). If there exists at least one $g$-convex
function that dominates $\varphi $, then $\varphi $ is $g$-convex if
and only if it is represented as the supremum of all $g$-convex
$C^{2}$-functions that dominate $\varphi $.
\end{proposition}

Motivated by Proposition \ref{thm-g-convexity-nonsmooth0} and the
discussions about abstract convexity in \cite{pallaschkeR1997} or \cite%
{singer1997}, we can find $g$-convex functions by another way.

For given $g$, we define
\begin{equation*}
\mathbf{\Pi }_{g}^{v}=\left \{ (a,b)\in \mathbb{R}^{2}:g(t,ay+b,az)\geq
ag(t,y,z),\  \forall y,z,\ dP\times dt\text{-}a.s.\right \}
\end{equation*}

It is clear that that $\mathbf{\Pi}_g^v$ cannot be empty, at least it
contains a element $(1,0)$, and if $g=\left<\xi_t,z\right>$ where $(\xi_t)_{t%
\in[0,T]}$ is a $\mathbb{R}^d$-valued progressively measurable process, $%
\mathbf{\Pi}_g^v=\mathbb{R}^2$. For each $(a,b)\in \mathbf{\Pi}_g^v$, $%
h(y)=ay+b$ is an affine $g$-convex function.

\begin{proposition}
\label{thm-g-convexity-nonsmooth1} Let $g$ satisfy (\ref{eqn2.2}) and $\phi
\in C(\mathbb{R})$. Then
\begin{equation*}
f(y)=\sup \left \{ h(y)=ay+b:\forall (a,b)\in \mathbf{\Pi }_{g}^{v}\text{
such that }h\leq \phi \right \} .
\end{equation*}%
is $g$-convex.
\end{proposition}

\begin{proof}
The proof is similar to that of Proposition \ref{thm-g-convexity-nonsmooth0}.
\end{proof}

\begin{remark}
From the above theorem, it follows that for each $(a,b)\in \mathbf{\Pi }%
_{g}^{v}$,
\begin{equation*}
\mathbb{E}_{t,T}^{g}[aX+b]\geq a\mathbb{E}_{t,T}^{g}[X]+b.
\end{equation*}%
But we cannot change the sign "$\geq $" to "$=$" in general although $%
h(y)=ay+b$ is an affine function, because $h$ here may be not a $g$-affine
function.
\end{remark}

The following property is easy to be proved:

\begin{proposition}
Let $g$ satisfy (\ref{eqn2.2}) and let $h$ and $\psi $ be two continuous
functions. Then

(i) If $\psi $ is $g$-affine and $h$ is $g$-convex, then $h\circ \psi $ is $%
g $-convex.

(ii) If $h$ is $g$-convex and increasing, and $\psi $ is $g$-convex, then $%
h\circ \psi $ is $g$-convex.
\end{proposition}

We also have the following stability property for $g$-convex functions.

\begin{theorem}
Let $g$ satisfy (\ref{eqn2.2}) and the $g$-convex (resp. concave) functions $%
h_{k}:\mathbb{R}\rightarrow \mathbb{R}$ converge pointwise for $k\rightarrow
\infty $ to $h:\mathbb{R}\rightarrow \mathbb{R}$. Then $h$ is $g$-convex
(resp. concave) and, for each compact set $S\in \mathbb{R}$, the convergence
of $h_{k}$ to $h$ is uniform on $S$.
\end{theorem}

\begin{proof}
Convexity of $h$ is trivial since $h_{k}$ is convex. And for each compact
set $S\in \mathbb{R}$, the convergence of $h_{k}$ to $h$ is uniform on $S$
(See \cite[pp. 177, Theorem 3.1.5]{hiriart-urrutyL1991}).

We now prove that $h$ is $g$-convex. Given bounded $\mathcal{F}_{T}$%
-measurable random variable $X$, we assume $\left \vert X\right \vert \leq M$%
. The uniform convergence means that there exists a function $\delta _{M}(k)$
with $\delta _{M}(k)\rightarrow 0$ as $k\rightarrow \infty $ such that for
each $x\in B[0,M]$ we have
\begin{equation*}
\left \vert h_{k}(x)-h(x)\right \vert \leq \delta _{M}(k).
\end{equation*}%
This implies that $h_{k}(X)\rightarrow h(X)$ in $L^{2}$ as $k\rightarrow
\infty $. Therefore by the continuity of $\mathbb{E}^{g}[\cdot ]$, we have
\begin{equation*}
\mathbb{E}_{t,T}^{g}[h(X)]\geq h(\mathbb{E}_{t,T}^{g}[X]),\quad P-a.s.\mbox{
for $t\in[0,T]$}.
\end{equation*}%
This with Lemma \ref{bound} it follows that for each $X\in L^{2}(\mathcal{F}%
_{T})$ such $h(X)\in L^{2}(\mathcal{F}_{T})$,
\begin{equation*}
\mathbb{E}_{t,T}^{g}[h(X)]\geq h(\mathbb{E}_{t,T}^{g}[X]),\quad P-a.s.\mbox{
for $t\in[0,T]$}.
\end{equation*}%
Thus $h$ is $g$-convex.
\end{proof}

\subsection{Some interesting properties of $g$-convexity}

\label{subsec7.2}

%In this subsection, we will discuss more properties of $g$%
%-convexity. Let us start by proving a property of $g$-convex
%function.

As mentioned before, for given $g$, the set of all $g$-convex functions is a
subset of that of convex functions. From Corollary \ref{thm-g-convex-convex}
and Hu's result in \cite{hu2005} (see also Corollary \ref{corhu}) it follows
that if $g$ is not super-homogeneous, then this inclusion is strict.

Unlike the classical situation, in general $h$ is $g$-convex does not
implies that $-h$ is $g$-concave. Let us consider the following example.

\begin{example}
\label{ex-01} Let $g=\left \vert z\right \vert$, the following statements
are equivalent:

(i) A function $h$ is $g$-convex;

(ii) $h$ is convex \emph{(}$h^{\prime \prime }(y)\geq 0$\emph{\ a.e.).%
\newline
}Moreover the following statements are also equivalent:

(iii) A function $h$ is $g$-concave;

(iv) $h$ is concave \emph{(}$h^{\prime \prime }(y)\leq 0$\emph{\ a.e.)} and
nondecreasing \emph{(}$h^{\prime }(y)\geq 0$\emph{\ a.e.).}
\end{example}

%
%We also have
%
%\begin{proposition}
%Let a function $h$ be continuous. Then the following two statements
%are equivalent:
%
%(i). $h$ is $g$-convex (resp. $g$-concave, $g$-affine) for each $g$
%satisfying (\ref{eqn2.2});
%
%(ii). $h(y)=y$ for each $y\in \mathbb{R}$.
%\end{proposition}
%
%\begin{proof}
%We only prove the case of $g$-convexity. (ii) $\Rightarrow$ (i) is
%obvious.
%
%(i) $\Rightarrow$ (ii): First, we assume that $g$ satisfies (\ref{eqn2.3}%
%-(b)). It follows from Theorem \ref{thm-g-convex-convex} that $h$ is
%convex.
%
%For any given $a\in \mathbb{R}$, assume that $g(t,y,z)=a$. By Definition (%
%\ref{defng-convexity}), we have for each $X\in L^{2}(\mathcal{F}_{T})$,$h(%
%\mathbb{E}_{t,T}^{g}[X])\leq \mathbb{E}_{t,T}^{g}[h(X)]$. Here
%\begin{equation*}
%\mathbb{E}_{t,T}^{g}[X]=X+\int_{t}^{T}a\,ds-\int_{t}^{T}z_{s}\,dW_{s}=E[X|%
%\mathcal{F}_{t}]+a(T-t).
%\end{equation*}%
%Thus $h(E[X]+aT)\leq E[h(X)]+aT.$ Putting $X=y\in \mathbb{R}$, we
%have
%\begin{equation*}
%h(y+aT)\leq h(y)+aT,\mbox{for any $y,a\in \Real$}.
%\end{equation*}%
%From the convexity of $h$, it follows that $h_{\pm }^{\prime }(y)aT\leq aT,%
%\mbox{for any $y,a\in \Real$}.$ Obviously, $h$ is a differentiable
%function and $h^{\prime }(y)=1$ for any $y\in \mathbb{R}$. Some
%elementary arguments (such as in the next proposition, we take
%$g=ay$) gives us the result what is required.
%\end{proof}

The following property implies that a convex function may not be a $g$%
-convex one.

\begin{example}
In the case when $g=ay$ where $a\in\Real$, we have $\mathbb{E}_{t,T}^{g}[c]=ce^{a(T-t)}$ and $%
\mathbb{E}_{t,T}^{g}[h(c)]=h(c)e^{a(T-t)}$. If $h$ is $g$-convex then
\begin{equation*}
h(c)e^{a(T-t)}=\mathbb{E}_{t,T}^{g}[h(c)]\geq h(ce^{a(T-t)}).
\end{equation*}%
From this relation it is easy to find a convex $h$ which is not $g$-convex.
\end{example}

We consider the following self-financing condition:
\begin{equation*}
\mathbb{E}_{t,T}^g[0]\equiv 0,\quad \forall 0\le t\le T.
\end{equation*}

\begin{corollary}
\label{cor-self-financing} Let $g$ satisfy (\ref{eqn2.2}). Then the
following three statements are equivalent:

(i) $\mathbb{E}^g[\cdot]$ satisfies the self-financing condition;

(ii) $g$ satisfies (\ref{eqn2.3})--a;

(iii) The constant function $h\equiv 0$ is $g$-affine.
\end{corollary}

\begin{proof}
The proof of the equivalence between (i) and (ii) can be found in \cite[%
Proposition 3.7]{peng2006b}. The equivalence between (ii) and (iii) follows
from Theorem \ref{thm-g-affinity} immediately.
\end{proof}

The "zero interest rate" condition means:
\begin{equation*}
\mathbb{E}_{t,T}^g[\eta]=\eta,\  \forall 0\le t\le T,\  \eta \in L^2(\mathcal{F%
}_t).
\end{equation*}

\begin{corollary}
\label{cor-zero-interest-rate} Let $g$ satisfy (\ref{eqn2.2}). Then the
following three statements are equivalent:

(i) $\mathbb{E}^g[\cdot]$ satisfies the zero interest rate condition;

(ii) $g$ satisfies (\ref{eqn2.3})--b;

(iii) For each constant $c$, the functions $h(y)=c$ and $y$ are $g$-affine.
\end{corollary}

\begin{proposition}
Let $g$ satisfy (\ref{eqn2.2}) and $c$ be a constant, then the following
statements are equivalent

(i) $\mathbb{E}_{t,T}^g[X+c]=\mbox{(resp. $\ge$, $\le$)}\mathbb{E}%
_{t,T}^g[X]+c$, for each $X\in L^2(\mathcal{F}_T)$,;

(ii) $g(t,y+c,z)=\mbox{(resp. $\ge$, $\le$)}g(t,y,z)$ for each $y\in \mathbb{%
R},z\in \mathbb{R}^d$.
\end{proposition}

\begin{remark}
(ii) means that $g$ is a periodic function in $y$ with period $c$.
\end{remark}

\begin{proof}
It is clear that the function $h(y)=y+c$ is $g$-affine (resp. $g$-convex, $g$%
-concave).
\end{proof}

\begin{corollary}
\label{cor-trans-invar} The following statements are equivalent

(i) $\mathbb{E}_{t,T}^g[X+c]=\mathbb{E}_{t,T}^g[X]+c$, for each $X\in L^2(%
\mathcal{F}_T)$, $c\in \mathbb{R}$;

(ii) $g$ is independent of $y$.

(iii) $h(y)=y+c$ is $g$-affine, $c\in \mathbb{R}$.
\end{corollary}

This result is a generalization of Lemma 3.2 in \cite{peng2004}. In
addition, it is clear that if, for each $c\in \mathbb{R}$, $h+c$ is $g$%
-convex implies $h$ is $g$-convex, then $g$ must be independent of $y$.


\begin{thebibliography}{Rosazza}
\bibitem[ALL]{alvarezLL1997} Alvarez, O., Lasry, J.-M. and Lions, P.-L.,
(1997) Convex viscosity solutions and state constraints, \emph{J. Math.
Pures Appl.} 76, pp. 265-288.

\bibitem[BCHMP]{briandCHMP2000} Briand, P., Coquet, F., Hu, Y., M\'{e}min,
J. and Peng, S., (2000) A converse comparison theorem for BSDEs and related
properties of g-expectation, \emph{Electron. Comm. Probab.} 5, pp.101-117.

\bibitem[B-El]{barrieuE2005} Barrieu, P. and El Karoui, N., (2005) Pricing,
hedging and optimally designing derivatives via minimization of risk
measures, Preprint, To appear in Volume on Indifference Pricing (ed: Rene
Carmona), Princeton University Press.

\bibitem[BQR]{buckdahnQR2000} Buckdahn, R., Quincampoix, M. and Rascanu, A.,
(2000) Viability property for a backward stochastic differential equations
and applications to partial differential equations, \emph{Probab. Theory
Relat. Fields} 116, pp.485-504.

\bibitem[CE]{chenE2002} Chen, Z. and Epstein, L., (2002) Ambiguity, risk
andasset returns in continuous time. \emph{Econometrica} 70, 1403--1443.

\bibitem[CKJ]{chenKJ2003} Chen, Z., Kulperger, R. and Jiang, L., (2003)
Jensen's inequality for $g$-expectation: part 1, \emph{C.R. Acad. Sci.
Paris, Ser. I} 333, pp.725-730.

\bibitem[CHMP]{coquetHMP2002} Coquet, F., Hu, Y., M\'{e}min, J. and Peng,
S., (2002) Filtration consistent nonlinear expectations and related $g$%
-expectations, \emph{Probab. Theory and Related Fields} 123, pp. 1-27.

\bibitem[CIL]{crandallIL1992} Crandall, M.~G., Ishii, H. and Lions, P.L., %
\newblock Users' guide to viscosity solutions of second order partial
differential equations. \newblock {\em Bull. Amer. Math. Soc.}, 27:1--67,
1992.

\bibitem[DE]{duffieE1992} Duffie, D. and Epstein, L., (1992) Stochastic
differential utility, \emph{Econometrica} 60(2), pp. 353-394.

\bibitem[EPQ]{elkarouiPQ1997} El~Karoui, N., Peng, S. and Quenez, M.C.,
(1997) Backward stochastic differential equations in finance, \textit{Math.
Finance} 7 (1), 1-71.

\bibitem[EQ]{elkarouiQ1997} El~Karoui, N. and Quenez, M.C., (1997) Nonlinear
pricing theory and backward stochastic differential equations, Biais, B.
(ed.) et al., Financial mathematics. Letures given at the 3rd session of the
Centro Internazionale Matematico Estivo (CIME), held in Bressanone, Italy,
July 8-13, 1996. Berlin: Springer. Lect. Notes Math. 1656, 191-246.

\bibitem[F-RG]{frittelliR2004} Frittelli, M. and Rossaza Gianin, E. (2004)
Dynamic convex risk measures, Szeg\"{o} (ed.) Risk Measures for the 21st
Century, Wiley-Finance, 227--247.

\bibitem[HL]{hiriart-urrutyL1991} Hiriart-Urruty, J.-B. and Lemar\'{e}chal,
C., (1991) Convex Analysis and Minimization Algorithms I, Springer-Verlag,
Berlin, Heidelberg, New York, London.

\bibitem[Hu]{hu2005} Hu, Y., (2005) On Jensen's inequality for $g$%
-expectation and for nonlinear expectation, \textit{Archiv der Mathematik},
85, 572-580, 2005.

\bibitem[PR]{pallaschkeR1997} Pallaschke, D. and Rolewicz, S., (1997)
Foundations of Mathematical Optimization, Kluwer Academic Publishers,
Dordrecht.

\bibitem[PP]{pardouxP1990} Pardoux, E. and Peng, S., (1990) Adapted Solution
of a Backward Stochastic Differential Equation, \textit{System and Control
Letters}, 14, 55-61.

\bibitem[P1995]{peng1997b} Peng, S., (1997) BSDE and Stochastic
Optimizations, Topics in Stochastic Analysis, Lecture Notes of 1995 Summer
School in Math. Yan, J., Peng, S., Fang, S., Wu, L.M. Ch.2, (Chinese vers.),
Science Press, Beijing.

\bibitem[P1997]{peng1997} Peng, S., (1997) BSDE and related $g$-expectation,
in \emph{Pitman Research Notes in Mathematics Series, no. 364, "Backward
Stochastic Differential Equations", Ed. by N. El Karoui and L. Mazliak},
141-159.

\bibitem[P1999]{peng1999} Peng, S., (1999) Monotonic limit theorem of BSDE
and nonlinear decomposition theorem of Doob-Meyer's type, \emph{Prob. Theory
Rel. Fields} \textbf{113}(4) 473-499.

\bibitem[P2004]{peng2004} Peng, S., (2004) Nonlinear expectation, nonlinear
evaluations and risk measurs, in K. Back T. R. Bielecki, C. Hipp, S. Peng,
W. Schachermayer, \emph{Stochastic Methods in Finance Lectures}, 143--217,
LNM 1856, Springer-Verlag.

\bibitem[P2004b]{peng2004b} Peng, S., (2004) Filtration Consistent Nonlinear
Expectations and Evaluations of Contingent Claims, \emph{Acta Mathematicae
Applicatae Sinica,} English Series \textbf{20}(2), 1--24.

\bibitem[P2005]{peng2005} Peng, S., (2005) Nonlinear expectations and
nonlinear Markov chains, \emph{Chin. Ann. Math.} \textbf{26B}(2) ,159--184.

\bibitem[P2005a]{peng2005a} Peng, S., (2004) Dynamical evaluations, \emph{C.
R. Acad. Sci. Paris,} Ser.I \textbf{339} 585--589.

\bibitem[P2006a]{peng2006a} Peng, S., (2006) $G$--Expectation, $G$--Brownian
Motion and Related Stochastic Calculus of It\^{o}'s type, preprint (\emph{%
pdf-file available in arXiv:math.PR/0601035v1 3Jan 2006}), to appear in
\emph{Proceedings of the 2005 Abel Symposium}.

\bibitem[P2005b]{peng2005b} Peng, S., (2005), Dynamically consistent
nonlinear evaluations and expectations, preprint (\emph{pdf-file available
in arXiv:math.PR/0501415 v1 24 Jan 2005}).

\bibitem[P2006b]{peng2006b} Peng, S., (2006) Modelling Derivatives Pricing
Mechanisms with Their Generating Functions, preprint (\emph{pdf-file
available in arXiv:math.PR/0605599v1 23 May 2006}).

\bibitem[Rosazza]{rosazza2005} Rosazza-Gianin, E. G., (2004) Risk measures
via $g$-expectations, \emph{Insurance: Mathematics and Economics} \textbf{36}%
(1) 19--34.

\bibitem[Singer]{singer1997} Singer, I., (1997) Abstract Convex Analysis,
Wiley-Interscience Publication, New York.

\bibitem[Yong]{yong1999} Yong, J., (1999) European-type contingent claims in
an incomplete market with constrained wealth and portfolio, \emph{%
Mathematical Finance} \textbf{9}(4) 387--412.
\end{thebibliography}
\end{document}